\documentclass[12pt]{amsart}
\usepackage{amscd}
\usepackage{verbatim}
\usepackage{amssymb, amsmath, amsthm, amscd,ifthen}
\usepackage[dvips]{graphics}
\usepackage[cp866]{inputenc}
\usepackage{graphicx}
\usepackage{epsfig}
\usepackage{xcolor}

\usepackage{amsmath,amssymb,amscd,}
\usepackage[all]{xy}

\usepackage{amssymb}

\newtheorem{thm}{Theorem}[section]

\newtheorem{prop}[thm]{Proposition}

\newtheorem{Ex}[thm]{Example}

\newtheorem{lemma}[thm]{Lemma}
\newtheorem{cor}[thm]{Corollary}

\theoremstyle{definition}
\newtheorem{rem}[thm]{Remark}
\newtheorem{dfn}[thm]{Definition}

\newtheorem{conj}[thm]{Conjecture}

 \DeclareMathOperator{\Prism}{Prism}
  \DeclareMathOperator{\dist}{dist}
  \DeclareMathOperator{\Id}{Id}

\usepackage{tikz}
\usepackage{tkz-euclide}
\usetikzlibrary{shapes,arrows,spy,positioning,snakes}

\newtheorem{problem}[thm]{Problem}

\title[Splitting necklaces]{Splitting necklaces, with constraints}

\author[D. Joji\'{c}]{Du\v{s}ko Joji\'{c}}
\author[G. Panina]{Gaiane Panina}
\author[R. \v{Z}ivaljevi\'{c}]{Rade \v{Z}ivaljevi\'{c}}

\address[D. Joji\'{c}]{ Faculty of Science, University of Banja Luka}
\address[G. Panina]
{Mathematics \& Mechanics Department, St. Petersburg State University; St. Petersburg Department of Steklov Mathematical Institute}
\address[R. \v{Z}ivaljevi\'{c}]{Mathematical Institute SASA, Belgrade}

\address{
  }


\begin{document}

\begin{abstract}
 We prove several versions of {\em  N. Alon's necklace-splitting theorem}, subject to additional constraints, as illustrated by the following results.

\smallskip
  (1)
 The ``almost equicardinal necklace-splitting theorem''  claims that, without increasing the number of cuts, one guarantees the existence of a fair splitting such that each thief is
 allocated (approximately) one and the same number of pieces of the necklace (including ``degenerate pieces'' if they exist), provided the number of thieves  $r=p^\nu$ is
 a prime power.

 (2) The ``binary splitting theorem'' claims that if $r=2^d$ and the thieves are associated with the vertices of a $d$-cube then, without increasing the number of cuts, one can guarantee the existence of a fair splitting such that adjacent pieces are allocated to thieves that share an edge of the cube.

 This result provides a positive answer to the ``binary splitting necklace conjecture''  of Asada at al.  (Conjecture 2.11 in \cite{AsadaFrick17}) in the case $r=2^d$.

 (3) An interesting variation arises when the thieves have their own  individual preferences. We prove several ``envy-free fair necklace-splitting theorems'' of various level of generality. By specialization we obtain numerous corollaries, among them  envy-free versions of (a) ``almost equicardinal splitting theorem'', (b) ``necklace-splitting theorem for $r$-unavoidable preferences'', (c) ``envy-free binary splitting theorem'', etc.  As a corollary we also obtain a recent result of Avvakumov and Karasev \cite{AK} about envy-free divisions where players may prefer an empty part of the necklace.
\end{abstract}

 \maketitle \setcounter{section}{0}

\section{Introduction}

The {\em Splitting Necklace Theorem} of N. Alon \cite{Alo87,alon-constructive} is one of the best known early results of topological combinatorics where the methods
of algebraic topology were applied with great success. The name of the theorem stems from the interpretation of the interval $[0,1]$ as an open (unclasped), continuous necklace where $n$
probability measures $\mu_i$ describe the distribution of ``precious gemstones'' of $n$ different type.
The result, together with its discrete version \cite{Alo87, alon-discr},  solves the problem of finding the minimum number of the cuts of the necklace needed for a fair distribution
of pieces among $r$ persons ($r$ ``thieves'' who stole the necklace).

\begin{thm}\label{thm:Alon} {\rm (\cite{Alo87})} Let $\mu_1,\mu_2,\ldots,
\mu_n$ be a collection of $n$ (absolutely) continuous probability measures on
$[0,1]$. Let $r\geq 2$ and $N:=(r-1)n$. Then there exists a
partition of $[0,1]$ by $N$ cut points into $N+1$ intervals
$I_1,I_2,\ldots,I_{N+1}$ and a function $f : \{1,2,\ldots,
N+1\}\rightarrow \{1,\ldots, r\}$ such that for each $\mu_i$ and
each $j\in\{1,2,\ldots,r\}$,
$$
\sum_{p\in f^{-1}(j)} \mu_i(I_p) = 1/r \, .\ \ \ \ (\ast)
$$
\end{thm}

The ``\textit{fair splitting condition}''  $(\ast)$ illustrates the requirement that each thief should be treated fairly and receive an equal net value of the necklace, as evaluated by each of the measures $\mu_i$.
Theorem \ref{thm:Alon} is optimal, as far as the number of cuts is concerned. Indeed, if the measures are supported by pairwise disjoint intervals than all $(r-1)n$ cuts are necessary.


\medskip
It is an interesting question if the {\em necklace-splitting theorem} can be refined by adding extra conditions (constraints) on how  the pieces of the necklace are distributed among the thieves. These conditions may reflect an objective requirement, as in ``almost equicardinal necklace-splitting theorem'' (Theorem \ref{thm:necklace-balanced}), where in addition to $(\ast)$ the thieves would like to have (approximately) the same number of the pieces of the necklace.

\medskip
More subjective and possibly antagonistic   conditions typically  appear in \textit{``envy-free division'' }of some resource  where each of the players (agents, thieves) has personal preferences that should be taken into account and,  after the resource is divided and allocated, no player should be envious to his companions.

\medskip
The classical ``Equilibrium Theorem'' of Stromquist \cite{Strom} and Woodall \cite{Wood} (see also D. Gale \cite{G}) serves as a main example of a result about envy-free divisions of a ``plain necklace'' (modeled as the interval $[0,1]$ without measures). The emphasis in this result is on individual preferences  of each of the players (thieves), participating in the division of the necklace, and it is an interesting question if Theorem \ref{thm:Alon} can  be refined with conditions of this type.

\subsection{Splitting necklaces with additional constraints}
\label{sec:add-constraints}

Additional constraints in the necklace splitting problem (and in the general Tverberg problem) were originally introduced and studied in \cite{VuZi-93}. The emphasis in this and in a subsequent paper \cite{Hell-1}  was on finding good lower bounds on the number of distinct fair splittings of a generic necklace.

 A more recent paper \cite{AsadaFrick17} links the necklace splitting problem with other fair division problems and emphasizes the importance of the so called ``binary splitting of necklaces'' for studying the equipartitions of mass distributions by hyperplanes.

Our central new results are necklace splitting theorems with constraints of the following three types:

\begin{enumerate}
  \item \textit{``Almost equicardinal splitting''} (Theorem~\ref{thm:necklace-balanced}, Corollary~\ref{CorEquicard}).
  Assuming that $r$ is a prime power we show (Theorem~\ref{thm:necklace-balanced}) that, with the same number of cuts $N:=(r-1)n = kr+s-1 \, (0\leq s<r)$ as in the original Alon's theorem, it is always possible to
  fairly divide the necklace such that each of the thieves is given at most $k+1$ pieces and at most $s$ of them get exactly $k+1$ pieces. In the special case when $s=0$ and  $N+1$ is divisible by $r$ we obtain as a corollary the result
  that there exists an {\em equicardinal}  fair splitting of the necklace when each thief is given exactly the same number $k:= (N+1)/r$ of pieces.

 An interesting feature of the proof is that we initially use a larger number of cuts. Eventually we get rid of superfluous cuts and end up with the desired number $N:=(r-1)n$. Unlike the original Alon's theorem we need the condition that $r = p^\nu$ is a prime power and it remains an interesting open problem if this condition can be relaxed.

  \item \textit{``Binary splitting''} (Theorem~\ref{ThmBinarySpl}, Conjecture~\ref{conj:AsasaFrick}).
   Suppose that  $r=2^d$ and assume that thieves are positioned at the vertices of the
$d$-dimensional cube. A \textit{binary necklace splitting}  is a fair splitting with $N = (r-1)n$ cuts with the additional constraint that adjacent (possibly degenerate) pieces of the necklace are allocated to thieves
whose vertices share an edge.

A binary necklace splitting theorem is proven in \cite{AsadaFrick17}  for $d=2$, that is for the case of $4$ thieves. The idea of the  proof was to embed the necklace into the Veronese (moment) curve, and apply an equipartition result by two hyperplanes, which turns out to be a  binary splitting.

 We prove the existence of a binary splitting (Theorem \ref{ThmBinarySpl}) for
any $r=2^d$ and  with the same number of cuts $N:=(r-1)n$, by applying a more direct combinatorial/topological argument.

\item \textit{``Fair, envy-free splitting'' and ``almost equicardinal fair envy-free splitting'' theorems} (Theorems \ref{Thm1} and \ref{Thm2}) unify
    relatives of Alon's result with ``equilibrium results'' \cite{AK, G, MeZe, S-H, Strom, Wood} from mathematical economics. It turns out that adding a collection of individual preferences is similar to adding an extra measure, however this modification leads to substantially more general results. For illustration, Theorem  \ref{Thm1} says that if $r$ is a prime power, and the thieves have their own subjective preferences, there exists a {\em fair, envy-free splitting} of a  necklace with $n$ measures, which needs at most $(r-1)n +(r-1)$ cuts (i.e. $r-1$ additional cuts are sufficient  to make the division  envy-free).
Moreover, with the same number $(r-1)n +(r-1)$  of cuts, one can make the division almost equicardinal (Theorem \ref{Thm2}). We show (Theorem \ref{thm:AK4.1}) that a recent result of Avvakumov and Karasev \cite[Theorem 4.1]{AK} is a special case of Theorem \ref{Thm2} for $n=0$. It follows, as a consequence of \cite[Theorem 4.3]{AK}, that Theorem \ref{Thm2} is not true in general if $r$ is not a prime power. By other specializations (variations) of Theorems \ref{Thm1} and \ref{Thm2} we obtain numerous corollaries, among them  envy-free versions of (a) ``almost equicardinal splitting theorem'', (b) ``necklace-splitting theorem for $r$-unavoidable preferences'', (c) ``envy-free binary splitting theorem'', etc.
\end{enumerate}

\subsection{Fair splitting of a discrete necklace}
\label{sec:discrete}

The following theorem is referred to as the discrete necklace-splitting theorem.

\begin{thm}\label{thm:Alon-discrete} {\rm (\cite{Alo87})}
Every unclasped necklace with $n$ types of beads and $ra_i$ beads of type $i\in [n]$ has a fair splitting among $r$ thieves with at most $(r-1)n$ cuts.
\end{thm}

Theorem~\ref{thm:Alon-discrete} is a consequence of Theorem~\ref{thm:Alon} by an elementary combinatorial argument, see \cite[p. 249]{Alo87} \cite[Lemma 7]{alon-discr} or \cite[Lemma 2.3]{AsadaFrick17}.  By a similar argument most of
continuous necklace-splitting theorems with additional constraints are expected to have an obvious discrete version.  In particular this is the case for the results mentioned in Section~\ref{sec:add-constraints} (parts (1) and (2)). This is certainly not the case  with the envy-free splittings in general, however it may be true for some special classes of  preferences.

\bigskip

\noindent\textbf{Acknowledgements.}
We are grateful to the referees for very kind remarks and suggestions, pointing in particular to new consequences of our results (Theorem \ref{thm:AK4.1}). 
 This research was supported through the programme ``Research in Pairs'' by the Mathematisches Forschungsinstitut Oberwolfach in 2019.
Gaiane Panina was supported by the RFBR Grant 20-01-00070. Rade \v Zivaljevi\' c was supported, through the grants to Mathematical Institute SASA (Belgrade), by the Ministry of Education, Science and Technological Development of Serbia.

\section{Preliminaries and main definitions}

\subsection{Partition/allocation of a necklace}\label{subsec:allocation}

A partition of a necklace $[0,1]$ into $m = N+1$ parts is described by a sequence of cut points $$0 = x_0 \leq x_1\leq x_2\leq \dots\leq x_N\leq x_{m} = 1 \, .$$
(Here and in the sequel, $m=N+1$.)

The associated, possibly degenerate intervals $I_j := [x_{j-1}, x_j]\, (j=1,\dots, m)$ are distributed among the thieves by an {\em allocation function} $f : [m] \rightarrow [r]$. The
pair $(x,f)$, where $x = (x_1,x_2,\dots, x_N)$ is the sequence of cuts, is called a {\em partition/allocation} of a necklace.

\subsection{Fair and $(k,s)$-equicardinal partitions/allocations}\label{subsec:fair}

\begin{enumerate}
\item
A {\em partition/allocation} $(x,f)$ of a necklace is {\em fair} if each measure is evenly distributed among the thieves, i.e. if for each measure $\mu_j$ and each thief $i\in [r]$,
\[
\mu_j(\bigcup_{\nu\in f^{-1}(i)} I_\nu) = \frac{1}{r}.
\]
\item
A partition/allocation $(x,f)$ is $(k,s)$-\textit{equicardinal} if  \newline (i)
 each thief gets no more than $k+1$ parts (intervals), and (ii)
 the number of thieves receiving exactly $k+1$ parts  is not greater than $s$.
\end{enumerate}
Note that for an equicardinal  fair division it is not important
where we allocate the degenerate (one-point) segments. Actually,
in our setting for almost equicardinal necklace-splitting, we
prefer (Section 3) not to allocate them at all. However, we will
use degenerate segments in  binary splittings (Section \ref{sec:binary}).

\subsection{$r$-unavoidable and collectively unavoidable complexes}\label{sec:collectively}

\begin{dfn}\label{dfn:r-unavoidable}(\cite{bfz}, \cite{bestiary})
A simplicial complex $K\subseteq 2^{[m]}$ is {\em
$r$-unavo\-id\-able} if for each collection
$(A_{1},\ldots,A_{r})$ of pair-wise disjoint sets in $[m]$ there
exists $i$ such that $A_i\in K$.
\end{dfn}

Collectively unavoidable $r$-tuples of complexes are  introduced in \cite{jnpz}.
They were originally studied as a common
generalization of pairs of Alexander dual complexes, Tverberg unavoidable complexes of \cite{bfz} and $r$-unavoidable
complexes from \cite{bestiary}.

\begin{dfn}\label{dfn:pigeonhole}
An ordered $r$-tuple $\mathcal{K} = \langle K_1,\ldots,K_r\rangle$
of subcomplexes of $2^{[m]}$ is {\em collectively
$r$-unavo\-id\-able} if for each ordered collection
$(A_{1},\ldots,A_{r})$ of pair-wise disjoint sets in $[m]$ there
exists $i$ such that $A_i\in K_i$.
\end{dfn}

\subsection{Balanced simplicial complexes}

\begin{dfn}\label{dfn:balanced}
We say that a simplicial complex $K\subseteq 2^{[m]}$ is {\em $(m,k)$-balanced} if it is
positioned between two consecutive skeleta of the simplex on $m$ vertices,
\begin{equation}\label{eqn:binoms}
   {{[m]}\choose{\leqslant k}} \subseteq K \subseteq {{[m]}\choose{\leqslant k+1}}\, .
\end{equation}
\end{dfn}

\subsection{Borsuk-Ulam theorem for fixed point free actions}\label{subsec:Volovikov}

\begin{thm}\label{thm:Vol96} {\rm (Volovikov \cite{Vol96-1})} Let
$p$ be a prime number and $G = (\mathbb{Z}_p)^k$ an elementary
abelian $p$-group. Suppose that $X$ and $Y$ are fixed-point free
$G$-spaces such that $\widetilde{H}^i(X, \mathbb{Z}_p)\cong 0$ for
all $i\leq n$ and $Y$ is an $n$-dimensional cohomology sphere over
$\mathbb{Z}_p$. Then there does not exist a $G$-equivariant map
$f: X\rightarrow Y$.
\end{thm}

\subsection{Connectivity of symmetrized deleted joins  }\label{sec:symmetrized}

\begin{dfn}
The {\em deleted join} \cite[Section~6]{Mat} of a family
$\mathcal{K} = \langle K_i\rangle_{i=1}^r = \langle K_1,\dots,
K_r\rangle$ of subcomplexes of $2^{[m]}$ is the complex
$\mathcal{K}^\ast_\Delta = K_1\ast_\Delta\dots \ast_\Delta   K_r
\subseteq (2^{[m]})^{\ast r}$ where  $A = A_1\sqcup\dots\sqcup A_r\in
\mathcal{K}^\ast_\Delta$ if and only if $A_j$ are pairwise
disjoint and $A_i\in K_i$ for each $i=1,\dots, r$. In the case $K_1 =
\dots = K_r = K$ this reduces to the definition of $r$-fold
deleted join $K_\Delta^{\ast r}$, see \cite{Mat}.

The {\em symmetrized deleted join} \cite{jvz-2} of   $\mathcal{K}$ is defined as
\[
SymmDelJoin(\mathcal{K}) := \bigcup_{\pi\in S_r}
K_{\pi(1)}\ast_\Delta\dots \ast_\Delta   K_{\pi(r)}\subseteq
(2^{[m]})^{\ast r}_\Delta \, ,
\]
where the union is over the set of all permutations of $r$ elements and
$(2^{[m]})_\Delta^{\ast r}\cong [r]^{\ast m}$   is the $r$-fold
deleted join of a simplex with $m$ vertices.
\end{dfn}
\noindent An element   $A_1\sqcup\dots\sqcup A_r \in
(2^{[m]})^{\ast r}_\Delta$ is from here on recorded as
$(A_1,A_2,\ldots,A_r;B)$ where $B$ is the complement of
$\cup_{i=1}^r~A_i$, so in particular $A_1\sqcup\dots\sqcup
A_r\sqcup B = [m]$ is a partition of $[m]$ such that
$A_i\neq\emptyset$ for some $i\in [r]$.

\begin{lemma}\label{lemmaDimension}{
The dimension of the simplex can be read of from $\vert B\vert$ as
$${\rm dim}(A_1,\ldots,A_r;B) = m - \vert B\vert -1.$$}
\end{lemma}

The following theorem is one of the two main results from \cite{jpz}.

\begin{thm}\label{thm:main-old} Suppose that $\mathcal{K} = \langle K_i\rangle_{i=1}^r = \langle K_1,\dots,
K_r\rangle$ is a collectively $r$-unavoidable family of subcomplexes of $2^{[m]}$. Moreover, we assume that there exists $k\geq 1$ such that $K_i$ is $(m,k)$-balanced
for each $i=1,\dots, r$. Then the associated symmetrized  deleted join
$$SymmDelJoin(\mathcal{K}) = SymmDelJoin(K_1,\dots, K_r)$$
is $(m-r-1)$-connected.
\end{thm}

The following theorem \cite[Theorem 3.3]{jvz-2} was originally proved by a direct shelling argument. As demonstrated in \cite{jpz} it can be also deduced from Theorem~\ref{thm:main-old}.

\begin{thm}\label{thm:jvz-2-symm}
Let $r,d\geq 2$ and assume that $rt+s = (r-1)d$ where $r$ and $s$ are the unique integers such that $t\geq 1$ and $0\leq s< r$. Let $N = (r-1)(d+2)$ and $m = N+1$.
Then the symmetric deleted join $SymmDelJoin(K_1,\dots, K_r)$
of the following skeleta  of the simplex $\Delta^N = 2^{[N+1]}$,
\begin{equation}\label{eq:balance-skeleta}
K_1 = \dots = K_s = {[N+1]\choose \leqslant t+2}, \quad  K_{s+1} = \dots = K_r = {[N+1]\choose \leqslant t+1} \, .
\end{equation}
is $(m-r-1)$-connected.
\end{thm}

\section{New configuration spaces for constraint splittings}
\label{sec:new-CS}

Perhaps the main novelty in our approach and the central new idea, emphasizing the role of {\em collectively unavoidable complexes},
is the construction and application of {\em modified (refined)} configuration spaces for splitting necklaces.

We begin by recalling a ``deleted join'' version of the {\em configuration space/test map scheme} \cite{Z17}, applied to the problem of
splitting necklaces, as described in \cite{VuZi-93} (see also \cite{Mat} for a more detailed exposition).

\medskip

\subsection{Primary configuration space}
\label{sec:primary}

The configuration space of all sequences $0= x_0 \leq x_1\leq \ldots\leq
x_N\leq x_m = 1 \, (m = N+1)$ is an $N$-dimensional simplex $\Delta^N$, where the numbers $\lambda_j := x_j-x_{j-1}\, (j=1,\dots, m)$ play the role of barycentric coordinates.
For a fixed allocation function $f : [m] \rightarrow [r]$, the set of all partitions/allocations $(x,f)$ is also coordinatized as
a simplex $C_f \cong \Delta^N$.  The primary configuration space, associated to the necklace-splitting problem, is obtained by gluing
together $N$-dimensional simplices $C_f$, one for each function $f : [m] \rightarrow [r]$. Note that the common face of $C_{f_1}$ and $C_{f_2}$ is the set of all pairs
$(x,f_1)$  ($\sim (x,f_2)$) such that $I_j = [x_{j-1}, x_j]$ is degenerate if $f_1(j) \neq f_2(j)$.

The simplicial complex obtained by this construction turns out to be (the geometric realization of) the deleted join $\mathcal{C}=\mathcal{C}_{m,r}:=
(\Delta^N)^{\ast r}_\Delta \cong [r]^{\ast m}$. Indeed,  a simplex $\tau = (A_1, A_2,\dots, A_r; B) \in (\Delta^N)^{\ast r}_\Delta$ is described as a partition $A_1\sqcup A_2\sqcup \dots \sqcup A_r\sqcup B = [m]$, and a partition/allocation $(x,f)$ is in (the interior of the geometric realization of) $\tau$ if and only
if $B = \{j\in [m] \mid I_j = [x_{j-1}, x_j] \mbox{ {\rm is degenerate}}\}$  and $A_i = f^{-1}(i)\setminus B$ is the set of all non-degenerate intervals allocated to $i\in [r]$.

Moreover  a simplex $\tau = (A_1, A_2,\dots, A_r; B)$  is the common face of $C_{f_1}$ and $C_{f_2}$ if and only if $B = \{j \in [m] \mid f_1(j)\neq f_2(j)\}$ and  for each $i\in [r]$,  $A_i = f_1^{-1}(i) \setminus B = f_2^{-1}(i)\setminus B$.

The key property of the complex $\mathcal{C}=\mathcal{C}_{m,r} \cong [m]^{\ast r}$, important for applications to the splitting necklaces problem, is its high connectivity. By the connectivity of a join formula \cite[Section 4.4.3]{Mat}, $[r]^{\ast m}$ is a $(m-2)$-connected, $(m-1)$-dimensional simplicial complex.

\subsection{The test map for detecting fair splittings}\label{subsec:detecting}
\label{sec:test-map}

Let $\mu = (\mu_1,\dots, \mu_n)$ be the vector valued measure associated to the collection of measures $\{\mu_j\}_{j=1}^n$. If $(x,f) \in (A_1,\dots, A_r;B)\in [r]^{\ast m}$ is a partition/allocation
of the necklace let
$$\phi_i(x,f) := \mu(\bigcup_{j\in A_i} I_j) = \sum_{j\in A_i} \mu(I_j)\in \mathbb{R}^n$$
be the total $\mu$-measure of all intervals $I_j = [x_{j-1}, x_j]$, allocated to the thief $i\in [r]$. If $\phi(x,f) := (\phi_1(x,f),\dots, \phi_r(x,f))\in (\mathbb{R}^n)^r$ then
$(x,f)$ is a fair splitting if and only if $\phi(x,f)\in D$, where $D := \{(v,\dots, v) \mid v\in \mathbb{R}^n\} \subset (\mathbb{R}^n)^r$ is the diagonal subspace.

\noindent
Summarizing, $(x,f)\in (\Delta^N)^{\ast r}_\Delta$ is a fair splitting of the necklace $([0,1]; \{\mu_j\}_{j=1}^n)$ if and only if $(x,f)$ is a zero of the composition map
\begin{equation}\label{eqn:composite}
      \widehat{\phi} :  (\Delta^N)^{\ast r}_\Delta \longrightarrow (\mathbb{R}^n)^r/D \, .
\end{equation}

\subsection{The group of symmetries}

The final ingredient in applications of the configuration space/test map scheme is a group $G$ of symmetries \cite{Z17}, characteristic for the problem. In the chosen scheme it is the $p$-toral group
$G = (\mathbb{Z}_p)^\nu$, where $p$ is a prime and $r=p^\nu$. The group $G$
acts freely, as a subgroup of the symmetric group $S_r$, on the deleted join $(\Delta^N)^{\ast r}_\Delta$ and without fixed points on the sphere $S((\mathbb{R}^n)^r/D)\subset (\mathbb{R}^n)^r/D$. (If $V$ is $G$-vector space with a $G$-invariant (euclidean) metric, then the associated unit sphere $S(V) = \{x\in V \mid
 \| x\| =1 \}$ is also $G$-invariant.)

By construction the map (\ref{eqn:composite})  is $G$-equivariant.

\subsection{New (refined) configuration spaces}\label{sec:refined}

{For each constraint an adequate refined configuration space should be carefully designed. In principle the constraint dictates
the choice of an appropriate  configuration space, as a subspace of $(\Delta^N)^{\ast r}_\Delta$. However this choice may not be unique and even the initial choice of the parameter $N$ may depend on the constraint.}

\subsubsection*{Refined configuration spaces for the almost  equicardinal splitting problem}

In order to derive Alon's necklace-splitting theorem (Theorem~\ref{thm:Alon}) it is natural to choose $N$, the dimension of the primary configuration space $(\Delta^N)^{\ast r}_\Delta$,
to be equal to the expected number of cuts,  $N= (r-1)n$.

Our basic new idea is to allow (initially) a larger number of cuts, but to force some of these cut points to coincide, by an appropriate choice of the configuration space.
This is achieved by choosing a $G$-invariant, $(r-1)n$-dimensional subcomplex $K$ of  the primary configuration space $(\Delta^N)^{\ast r}_\Delta$, where $N$ is (typically) larger than the
number $(r-1)n$ of essential cut points.

\medskip
Our first choice for a {\em refined configuration space} $K\subseteq (\Delta^N)^{\ast r}_\Delta$ is the symmetrized deleted join  $SymmDelJoin(\mathcal{K})$ of a family $\mathcal{K} = \{K_i\}_{i=1}^r$ of
collectively unavoidable subcomplexes of $2^{[m]}$ where $m = N+1 = (r-1)(n+1)+1$.

\subsubsection*{Refined configuration spaces for the binary splitting problem}
 For the binary splitting theorem we choose the usual parameter $N=(r-1)n$, as in Alon's original theorem.
Recall that maximal simplices (facets) of the primary configuration space $(\Delta^N)^{\ast r}_\Delta \cong [r]^{\ast (N+1)} $
can be interpreted as the graphs $\Gamma(f)\subset [N+1]\times [r]$ of functions $f : [N+1] \rightarrow [r]$.

Assuming that $r=2^d$ thieves are positioned on the vertices of a $d$-dimensional cube, we consider a subcomplex    $K \subset (\Delta^N)^{\ast r}_\Delta $ that includes the graphs of functions corresponding to binary splitting of the necklace.
More explicitly $\Gamma(f) \in K$ if and only if for each $i\in [N]$ either $f(i) = f(i+1)$ or $\{f(i), f(i+1)\}$ is an edge of the $d$-cube.

\section{Almost equicardinal necklace-splitting theorem}
\label{sec:Balanced-splittings}

We begin with a very simple example of a necklace where all fair partitions/allocations are easily described.

\begin{Ex}\label{ex:motivation}{\rm
Assume that the measures $\mu_j\, (j=1,\dots, n)$ are supported by pairwise disjoint subintervals of $[0,1]
$.   In this case we need at least $(r-1)n$ cuts which dissect the necklace into $(r-1)n+1$ parts.  We observe that for this
choice of measures there always exists a $(k,s)$-equicardinal, fair partition/allocation of measures to $r$ thieves where $k$ is the quotient and $s$ the corresponding remainder, obtained after the division of $(r-1)n+1$ by $r$. }
\end{Ex}

The choice of measures in Example~\ref{ex:motivation} is rather special and it is natural to ask if such a partition is always possible.

\begin{problem}\label{prob:natural}{\rm
For a given collection $\{\mu_j\}_{j=1}^n$ of absolutely continuous measures on $[0,1]$ and $r$ thieves, is it always possible to find a fair, $(k,s)$-equicardinal partition/allo\-ca\-ti\-on of the necklace where $k$ and $s$ are chosen as in Example~\ref{ex:motivation}? }
\end{problem}

The following extension of the classical necklace theorem of Alon provides an affirmative answer to Problem~\ref{prob:natural}.

\begin{thm}\label{thm:necklace-balanced} {\rm (Almost equicardinal necklace-splitting theorem)}
For given positive integers $r$ and $n$, where $r = p^\nu$ is a power of a prime, let  $k=k(r,n)$ and $s=s(r,n)$ be the unique non-negative integers such that $(r-1)n+1 = kr+s$ and $0\leq s < r$. Then for any choice of $n$ continuous, probability measures on $[0,1]$ there exists a fair  partition/allocation of the associated necklace with $(r-1)n$ cuts which is also $(k,s)$-equicardinal in the sense that:

\noindent (1) \, each thief gets no more than $k+1$ parts (intervals);

\noindent (2) \, the number of thieves receiving exactly $k+1$ parts  is not greater than $s$.
\end{thm}

\begin{proof} As emphasized in Section~\ref{sec:refined}, the basic idea of the proof is to initially allow a larger number of cuts, and then to force some of these cuts to be superfluous by an appropriate choice of the configuration space.

Our choice for a {\em refined configuration space} is the symmetric deleted join $K:= SymmDelJoin(K_1,\dots, K_r)$ of the family $\mathcal{K} = \langle K_i\rangle_{i=1}^r$,
\begin{equation}\label{eq:balance-skeleta-2}
K_1 = \dots = K_s = {[N+1]\choose \leqslant k+1}, \quad  K_{s+1} = \dots = K_r = {[N+1]\choose \leqslant k}
\end{equation}
of subcomplexes of the simplex $\Delta^N = 2^{[N+1]}$, where $N = (r-1)(n+1)$, and
\begin{equation}\label{eqn:racun}
  m = N+1 = (r-1)(n+1)+1 = r(k+1)+s-1  \, .
\end{equation}

\medskip
By substituting $k = t+1$ and $n = d+1$ in Theorem~\ref{thm:jvz-2-symm} we observe that the complex $K$ is $(m-r-1)$-connected. By construction (Section~\ref{sec:new-CS}) a partition/allocation $(x,f)\in K$ corresponds to
a fair division if and only if $\widehat{\phi}(x,f)=0$, where $\widehat{\phi}$ is the test map described in equation (\ref{eqn:composite}).

Suppose, for the sake of contradiction, that a fair division $(x,f)$ does not exist. Then $\widehat{\phi}(x,f)\neq 0$ for each $(x,f)\in K$ and
there arises a $G$-equivariant map
\[
 \widehat{\phi} : K  \longrightarrow S(\mathbb{R}^{nr}/D)\stackrel{G}{\simeq} S^{(r-1)n-1}
\]
where $G = (\mathbb{Z}_p)^r$ and $S(V)$ is a $G$-invariant sphere in a $G$-vector space $V$. Since by (\ref{eqn:racun})
\[
    m-r-1 = [(r-1)(n+1)+1] - r-1 = (r-1)n-1
\]
this contradicts Volovikov's theorem (Theorem~\ref{thm:Vol96}). Hence a point $(x,f)\in K$, representing a fair division, must exist.

Let  $(A_1,\dots, A_r;B)$ be the simplex of $K$ containing $(x,f)$. Then, with a possible reindexing of thieves,  $(x,f)\in \tau = (A_1,\dots, A_r; B)$ where $\vert A_i\vert \leq k+1$ for $i=1,\dots, s$
and    $\vert A_j\vert \leq k$ for $j=s+1,\dots, r$. From here it immediately follows that $(x,f)$ describes a $(k,s)$ balanced partition/allocation of the necklace.
\end{proof}

\begin{cor}\label{CorEquicard} (Equicardinal necklace-splitting theorem)
In the special case $s=0$, or equivalently if $(r-1)n+1$ is divisible by $r$, Theorem~\ref{thm:necklace-balanced} guarantees the existence
of a fair partition/allocation which is {\em equicardinal} in the sense that each thief is allocated exactly the same number of pieces of the necklace.
Here we tacitly assume that the necklace is generic, i.e. that all $(r-1)n$ cuts are needed.
\end{cor}

\subsection*{Splitting necklaces and collectively unavoidable complexes}

Collectively unavoidable complexes $\mathcal{K} = \{K_i\}_{i=1}^r$,  introduced in \cite{jnpz}, include as a special case  pairs of Alexander dual complexes \cite{Mat} (in the case $r=2$) and {\em unavoidable complexes}
\cite{bfz,bestiary} (in the case $K_1=\dots= K_r$).  As shown in \cite{jpz}, they form a very useful tool for proving theorems of van Kampen-Flores type. Here we demonstrate that they also provide
a natural environment for necklace-splitting theorems with  additional constraints.

\medskip
Theorem \ref{thm:necklace-balanced} turns out to be a very special case of the following theorem where the constraints on the partition/allocation are ruled by a collectively unavoidable $r$-tuple of complexes.

\medskip
As in Theorem~\ref{thm:necklace-balanced}, we assume that $r = p^{\nu}$ is a power of a prime number and $m = N+1 = (r-1)(n+1)+1$.
Moreover,   $k=k(r,n)$ and $s=s(r,n)$ are the unique non-negative integers such that $(r-1)n+1 = kr+s$ and $0\leq s < r$.

\begin{thm}\label{thm:general2}
Let  $\mathcal{K} = \langle K_i\rangle_{i=1}^r = \langle K_1,\dots,
K_r\rangle$ be a sequence of subcomplexes of $2^{[m]}$ such that:
\begin{enumerate}
\item each complex $K_i$ is $(m,k)$-balanced, and
\item the sequence $\mathcal{K}$ is collectively unavoidable.
\end{enumerate}
Choose a collection $\{\mu_\kappa\}_{\kappa=1}^n$ of $n$ continuous, probability measures on $[0,1]$. Then for any company $\mathcal{T}$ of $r$ thieves there exists a fair partition/allocation $(x,f)\in SymmDelJoin(\mathcal{K})$
of the associated necklace with at most $(r-1)n$  cuts.
More explicitly, there exists a $(r-1)n$-dimensional simplex $$(A_1,\dots, A_r; B)\in SymmDelJoin(\mathcal{K})$$ and a partition/allocation $(x,f)\in (A_1,\dots, A_r; B)$ which is fair for $\mathcal{T}$, with any choice
of a bijection $\mathcal{T} \leftrightarrow [r]$.
\end{thm}

\proof The proof is similar to the proof of Theorem \ref{thm:necklace-balanced}, with an additional intermediate step allowing us to control the number of essential cut points.

\medskip
As expected we use Theorem~\ref{thm:main-old}, instead of Theorem~\ref{thm:jvz-2-symm}, which claims that, under the conditions of the theorem, the complex $K:= SymmDelJoin(\mathcal{K})$
is $(m-r-1)$-connected.  However, we refine the configuration space even more by selecting the $(m-r)$-dimensional skeleton $K^{(m-r)}$ of $K$ as the domain for our test
map $\widehat{\phi}$. The complex $K^{(m-r)}$ is also $(m-r-1)$-connected (since $\pi_j(K^{(m-r)}) = \pi_j(K) = 0$ for each $j\leq m-r-1$) and the condition ${\rm dim}(K^{(m-r)}) = (r-1)n$ guarantees that the number of superfluous cuts (indexed by $B$) is
at least $r-1$.    \qed

\section{Binary necklace splitting}\label{sec:binary}

Recall (see \cite{AsadaFrick17} or Section \ref{sec:add-constraints}) that if  $r=2^d$  thieves are positioned at the vertices of the $d$-dimensional cube then a \textit{binary necklace splitting}  is a fair splitting with the additional constraint that adjacent (possibly degenerate) pieces of the necklace are allocated to thieves whose vertices share an edge.

\medskip
Note that the cut points (all $N = (r-1)n$ of them) are linearly ordered, which induces a linear ordering on all $N+1$ intervals (including degenerate). As a consequence each interval (degenerate or not) has at most two neighbors.

\begin{thm}\label{ThmBinarySpl}{\rm (Binary necklace-splitting theorem)}
Given a necklace with $n$ kinds of beads and $r=2^d$ thieves,
there always exists a binary necklace splitting with $(r-1)n$ cuts.
\end{thm}

Note that the authors of \cite{AsadaFrick17} originally introduced a slightly more general binary necklace splitting where  $r$ thieves are placed at the vertices of a cube of dimension $\lceil\log_2r\rceil$
(allowing some vertices of the cube to remain unoccupied).

\medskip
Theorem \ref{ThmBinarySpl} provides an affirmative answer to the following conjecture (Conjecture 2.11 in \cite{AsadaFrick17}) in the case when $r$ is a power of two.

\begin{conj}\label{conj:AsasaFrick}
  Given a necklace with $n$ kinds of beads and $r\geq 4$ thieves, there exists a binary necklace splitting of size $(r-1)n$.
\end{conj}

Both versions of the binary necklace splitting are special cases of the graph-constraint or $G$-constraint necklace splitting, where $G = (V,E)$ is a connected graph on the set $V \cong [r]$ of thieves. 

\begin{dfn}\label{def:G-constraint}
Let $G = (V,E)$ be a connected graph where $V \cong [r]$ is the set of thieves. A necklace splitting is $G$-constraint if the corresponding {\em partition/allocation} $(x,f)$, where
$f : [m] \rightarrow [r]$ is the allocation function (Section \ref{subsec:allocation}), has the property that for each $i=1,\dots, m-1$ either $f(i) = f(i+1)$ or $\{f(i), f(i+1)\}\in E$ is
an edge of the graph $G$. A function $f$ satisfying this condition will be referred to as a $G$-constraint allocation function.
\end{dfn}
Note that if for some $f(i)=f(i+1)$ the cut-point $x_i$ is superfluous and can be removed from the necklace splitting.

The following $G$-constraint simplicial subcomplex $K_G^m \subseteq (\Delta^N)^{\ast r}_\Delta \cong [r]^{\ast m}$  of the primary configuration space $[r]^{\ast m}$  is a
natural choice for a configuration space suitable for studying the $G$-constraint splittings of a necklace.

\begin{dfn}\label{def:G-complex}
Let $G = (V, E)$ be a connected graph. If $r = \vert V \vert$ is the cardinality of $V$ then w.l.g.\ we may assume that $G = ([r],E)$.
 The $G$-constraint complex $K_G^m \subset (\Delta^N)^{\ast r}_\Delta \cong [r]^{\ast m}$ ($G$-complex for short) is defined as the union of all simplices $C_f$
(Section \ref{sec:primary}) where $f : [m]\rightarrow [r]$ is a $G$-constraint allocation function (Definition \ref{def:G-constraint}). More intrinsically, a $G$-constraint allocation function can be interpreted as a walk on the graph where in each step one moves to an immediate neighbor or remains at the same vertex of the graph.
\end{dfn}

The primary configuration space $[r]^{\ast m}$ can be interpreted as the order complex $\Delta(\Pi)$ of a poset $\Pi$ on the set $V \times [m] \cong [r]\times [m]$ where  $(x, i) \preccurlyeq (y,j)$ in $\Pi$ if and only if  either $(x,i) = (y,j)$ or $i< j$. This observation is an immediate consequence of the definition of the order complex $\Delta(P)$ as the simplicial complex of all chains in a poset $P$.

Similarly let  $\Pi_G^m$ be a subposet of $\Pi$, defined on the same set of vertices $V\times [m]$, where in $\Pi_G^m$
$$(x,i)\preccurlyeq (y,j)\Leftrightarrow  i \leqslant j \textrm{ and  } \dist(x,y)\leq j-i.$$
(The distance function $\dist(x,y)$  is the graph-theoretic distance, i.e.\ the smallest number
of edges in a path connecting the vertices $x, y\in V$.)

\medskip
By comparison of definitions we see that $K^m_G \cong \Delta(\Pi_G^m)$ is the order complex of the poset $\Pi_G^m$.

\begin{rem}{\rm  By construction $K^{m}_G$ is always a subcomplex of the standard (primary)
configuration space $(\Delta^N)^{\ast r}_\Delta$ (where $m = N+1$) and $K^{m}_G=(\Delta^N)^{\ast
r}_\Delta$ if $G$ is the complete graph $K_r$. If $r=2^d$ and $C^d$ is the vertex-edge graph of the $d$-dimensional cube, then $K^{m}_{C^d}$ is a proper configuration space for the binary necklace-splitting problem.
  }
\end{rem}

As expected, in the course of the proof of Theorem \ref{ThmBinarySpl} the main step is the proof that the complex $K^{N+1}_{C^d}$ is
$N-1$ connected. For an inductive proof of this fact we need the following definition.

\begin{dfn}
For a given graph $G = (V, E)$, where $V = \{v_i\}_{i=1}^r$, let  $\Prism(G)=(V',E')$ to be a new
graph with $V'=\{v_i^{(1)}\}_{i=1}^r\cup
\{v_i^{(2)}\}_{i=1}^r$, as the set of vertices. The vertices $v_i^{(1)}$ and $v_j^{(1)}$
(respectively, $v_i^{(2)}$ and $v_j^{(2)}$) share an edge in $\Prism(G)$ if and
only if $\{v_i, v_j\}\in E$. Moreover, the copies of the same vertex
$v_i^{(1)}$ and $v_i^{(2)}$ always share an edge in $\Prism(G)$.
\end{dfn}
 Note that by definition $C^{d+1}=\Prism(C^d)=\Prism^d(\hbox{one-vertex graph})$.

\begin{prop}\label{prop:G-connectivity}
Suppose that $G = (V,E)$ is a connected graph. If $K^m_G$ is $(m-2)$-connected for all $m\geq 2$, then $K^m_{\Prism(G)}$
is also $(m-2)$-connected for all $m\geq 2$.
\end{prop}
\begin{proof}
The proof is by induction on $m$. Suppose that $K_G^{m'}$ is $(m'-2)$-connected for all $m'\geq 2$. Since $G$ is connected, the graph $\Prism(G)$ is also connected and the proposition is true for $m=2$. Let $\Pi_1 = \Pi_{\Prism(G)}^m\setminus \{(v_i^{(1)}, m)\}_{i=1}^r$ and $\Pi_2 = \Pi_{\Prism(G)}^m\setminus \{(v_i^{(2)}, m)\}_{i=1}^r$
be two subposets of $\Pi_{\Prism(G)}^m$. By definition  $\Pi_1 \cup\Pi_2=\Pi_{\Prism(G)}^m$ and $\Pi_1 \cap\Pi_2 = \Pi_{\Prism(G)}^{m-1}$.
Let $\Delta_1 = \Delta(\Pi_1)$ and $\Delta_2 = \Delta(\Pi_2)$ be the associated order complexes. It is not difficult to see that $K_{\Prism(G)}^m = \Delta_1\cup \Delta_2$ and $K_{\Prism(G)}^{m-1} = \Delta_1\cap \Delta_2$.

Let us show that the complex $\Delta_1\cong \Delta_2$ has the same homotopy type as the complex $K_G^m = \Delta(\Pi_G^m)$.
Let $e_1:\Pi_G^m \rightarrow \Pi_1$ be the inclusion map  which maps $(v, j)$ to $(v^{(2)}, j)$ and let  $\rho_1:\Pi_1\rightarrow\Pi_G^m$ be a monotone map of posets defined by the formula:
 $$\rho_1(v^{(i)},j)=\left\{%
\begin{array}{ll}
    (v,j), & \hbox{if $i=2$;} \\
   (v,j+1), & \hbox{if $i=1$.} \\
\end{array}%
\right.    $$
These maps satisfy the relations:
\begin{enumerate}
  \item $\Id = \rho_1\circ e_1: \Pi_G^m\rightarrow \Pi_G^m$, and
  \item $e_1\circ\rho_1(x)\succcurlyeq x \ \ \ \forall x \in \Pi_1.$
\end{enumerate}
By the homotopy property of monotone maps, see D. Quillen \cite[Section 1.3]{quill}  (or Theorem 12 from \cite{Rade}), we conclude that
both $e_1$ and $\rho_1$ induce homotopy equivalences of the order complexes $\Delta_1$ and $K^m_{G}$.
Similarly, we have a homotopy equivalence  $\Delta_2 \simeq K^m_{G}$.

By the inductive assumption, $\Delta_1\cap
\Delta_2=K^{m-1}_{\Prism(G)}$ is $(m-3)$-connected and since both $\Delta_1$ and $\Delta_2$ are $(m-2)$-connected by the {\em Gluing Lemma} (see \cite[Lemma 10.3]{Bjo95})
the complex $K^m_{\Prism(G)}$ is also  $(m-2)$-connected.
\end{proof}

\begin{cor}\label{Corcub}
The complex $K^{N+1}_{C^d}$ is $(N-1)$-connected.
\end{cor}

In light of the discrete-to-continuous reduction, described in Section \ref{sec:discrete}, Theorem \ref{ThmBinarySpl} is a consequence of the following result.

\begin{thm}\label{ThmBinarySpl-cont}{\rm (Binary splitting of continuous necklaces)}
If the number of thieves is $r=2^d$, then for each $n$ continuous probability measures $\mu_1, \dots, \mu_n$ on $[0,1]$, representing the distribution of $n$  kinds of beads,
there exists a binary necklace splitting with $(r-1)n$ cuts.
\end{thm}
\proof The proof is similar to the proof of Theorem \ref{thm:necklace-balanced} with Corollary \ref{Corcub} playing the role of Theorem \ref{thm:jvz-2-symm}. \hfill $\square$

\subsection{Binary necklace splitting and equipartitions by hyperplanes}

The {\em Gr\" unbaum-Hadwiger-Ramos hyperplane mass partition problem} \cite{Ram96, Ziv15, BFHZ, BFHZ', vz15} is the question of finding the smallest dimension $d = \Delta(j,k)$ such that for every collection of $j$ masses (measurable  sets, measures) in $\mathbb{R}^d$ there exist $k$ affine hyperplanes that cut each of the $j$ masses into $2^k$ equal pieces.

\medskip
Asada et al.\ in \cite{AsadaFrick17} obtained (the continuous version of) Theorem \ref{ThmBinarySpl} in the case $r=4$ by embedding the necklace ($= [0,1]$) into the moment curve
$\{(t,t^2,\dots, t^D) \mid -\infty \le t \le +\infty\}\subset \mathbb{R}^D$ and using a necklace splitting arising from an equipartition of the necklace  by two hyperplanes in $\mathbb{R}^D$.

\medskip
These authors correctly observed that their approach would allow them to deduce  the general case $r = 2^d$ of Theorem \ref{ThmBinarySpl} from Ramos' conjecture \cite{Ram96} which says that each collection of
$n$ continuous measures in $\mathbb{R}^D$ admits an equipartition by $d$ hyperplanes, provided $D\geq \frac{n(2^d -1)}{d}$.

\medskip
Moreover, they claim (at the end of Section 2) that a partial converse is true, i.e.\ that Theorem \ref{ThmBinarySpl} is strong enough to establish Ramos' conjecture for measures concentrated on the moment curve.

\medskip
This is unfortunately not the case since there exist binary necklace splittings which do not arise from equipartitions by hyperplanes, as illustrated by Example~\ref{ex:binary}. The reason is that
hyperplane splittings have an additional property of being ``balanced'', due to the fact that each hyperplane contributes the same number of cuts.

\begin{Ex}\label{ex:binary}{\rm
Let $r=4$ and $n=2$ and suppose that the thieves $A, B, C, D$ are positioned in a cyclic order on the vertices of a square. By the necklace-splitting theorem of Alon a continuous necklace with two types of beads (two measures $\mu_1, \mu_2$) there exists a fair division with $(r-1)n = 6$ cuts. Assume that $\mu_1$ and $\mu_2$ are, as in Example \ref{ex:motivation}, uniform probability measures on two disjoint intervals $I$ and $J$. Suppose that this fair division arises from an equipartition by two planes $H_1$ and $H_2$ in $\mathbb{R}^3$. (The reader is recommended to draw the projection of the moment curve in the plane orthogonal to the line $L = H_1\cap H_2$, and to analyse possible  dissections of intervals $I$ and $J$.)

The interval $I$ is subdivided into subintervals $I_1, I_2, I_3, I_4$ (by cut points $x_1<x_2<x_3$), similarly $J$ is subdivided into $J_1, J_2, J_3, J_4$ by cut points $y_1<y_2<y_3$ .  By taking into  account that  each plane has at most three points in common with the moment curve, we observe that  $\{x_1, x_3, y_2\}\subset H_1$ and $\{x_2, y_1, y_3\}\subset H_2$  (or vice versa).
 Assume that the intervals $I_1, I_2, I_3, I_4$ are in this order allocated to thieves $A, B, C, D$.

From here we deduce that $A$ and $B$ (respectively $C$ and $D$) are on different sides of the hyperplane $H_1$. Similarly  $A$ and $D$ (respectively $B$ and $C$) are on different sides of the hyperplane $H_2$. The rest of the allocation is uniquely defined and reads as follows, $J_1\mapsto D, J_2\mapsto A, J_3\mapsto B, J_4\mapsto C$.

In turn this shows that the binary necklace splitting 
\[
  I_1\mapsto A,\, I_2\mapsto B,\, I_3\mapsto C,\, I_4\mapsto D    \quad J_1\mapsto D, J_2\mapsto C, J_3\mapsto B, J_4\mapsto A
\]
cannot be obtained from an equipartition by two hyperplanes.
}
\end{Ex}

\section{Envy-free and fair necklace splitting}
\label{sec:envy-free}

A division $X = X_1\uplus\dots\uplus X_r$ of a resource $X$ among $r$ players $\{P_i\}_{i=1}^r$ (agents, thieves) is {\em envy-free} if each player $P_i$
 has a preference relation $\preccurlyeq_i$ and $X_j \preccurlyeq_i X_i$ for each $i$ and $j$. Informally speaking, in an envy-free division each player feels that, from her   individual point of view,  her share is at least as good  as the share of any other player, and therefore  no player feels envy.

 \medskip
As in previous sections, the resource in our paper is a  necklace with $n$ (absolutely) continuous measures, which is supposed to be fairly divided among $r$ thieves by a smallest number of cuts possible. A new moment is that this division is expected to be both fair (with respect to the measures) and envy-free (from the view point of their individual preferences).

\medskip
Measures can be also interpreted as preferences (a thief always chooses a piece with the largest measure). However, the preferences in general can be of quite different nature.
For example (in the necklace splitting context) one thief may prefer one side of the necklace, the other prefers  a small number of segments in his share, the third wants the largest segment possible, etc.

\medskip
Our main results in this section (Theorems \ref{Thm1}, \ref{Thm2} and \ref{ThmBinarySpl-envy}) show that, under some natural assumptions, the conclusion of the splitting necklace theorem still holds if one of the measures is replaced by a set of individual preferences.

\subsection*{Envy-free divisions}

Recall that a partition/allocation  of a necklace, as introduced in Section \ref{subsec:allocation}, is a pair  $(x,f)$ where    $x$ records the cuts of the necklace while $f: [m] \rightarrow [r]$ is an allocation function describing the shares of each of the thieves.

In this section an element $(x,f)\in (\Delta^N)^{\ast r}_\Delta$ of the primary configuration space (Section \ref{sec:primary}) is interpreted as a {\em partition/preallocation}. By ``pre-allocation'' we mean that now the role of the function $f$ is to partition the set $\{[x_{k-1}, x_k]\}_{k=1}^m$ of all intervals into ``shares'' $f^{-1}(i)$ and to put them on display (say in $r$ different safes), so the thieves can evaluate them from the view point of their own subjective preferences.

  \begin{dfn}\label{def:reality} {\rm
    After a partition/preallocation $\hat{x}=(x,f)$ a thief sees in a safe (labeled by) $i$, a collection $U_i$ of non-degenerate intervals. More explicitly $U_i = \{[x_{k-1}, x_k]\, \vert\,  f(k) =i \mbox{ {\rm and} } x_{k-1}\neq x_k \}$.   An ``ordered family'' ($r$-tuple) of collections of intervals $\mathcal{U} = \mathcal{U}(\hat{x}) = \{U_i\}_{i=1}^r$ arising by this construction  will be referred to as {\em admissible} (more precisely $(m,r)$-admissible) family. It is not difficult to see that a family  $\mathcal{U}= \{U_i\}_{i=1}^r$ is admissible if and only if:
    \begin{enumerate}
      \item  The set $U_i $ is a (possibly empty) collection of subintervals of $[0,1]$ with disjoint interiors  $(i=1,\dots, r)$;
      \item  The union $\cup \mathcal{U} = U_1\cup\dots\cup U_r$ is a cover of $[0,1]$ by at most $m$ intervals with disjoint interiors.
    \end{enumerate}  }
  \end{dfn}

\begin{dfn}\label{def:preference} A \emph{preference} of a person (player, thief), for a given $(m,r)$-admissible family  $\mathcal{U}=\{U_1, \dots, U_r\}$, is a choice of one (or more) preferred sets $U_i$.  In general there are no other restrictions, for example the chosen set $U_i$  may be empty. Equivalently, a preference is a function selecting a non-empty subset $\mathcal{P}(\mathcal{U}) = \mathcal{P}(U_1,\dots,U_r) \subseteq [r]$ for each  admissible family $\mathcal{U}$.
\end{dfn}

A preference function typically arises if a player has a (pre)order relation $\preccurlyeq_\mathcal{U}$ associated with each admissible family $\mathcal{U}$. In that case by definition
$$ i \in \mathcal{P}(\mathcal{U})  \Leftrightarrow  (\forall k\in [r])\, k\preccurlyeq_\mathcal{U} i \, . $$

\begin{dfn}\label{def:ignores}
     A preference $\mathcal{P}(\mathcal{U}) = \mathcal{P}(U_1,\dots,U_r)$ is $G$-\emph{equivariant} if for any permutation $\sigma \in G\subseteq S_r$,
$$\sigma(\mathcal{P}(U_{\sigma (1)},\dots,U_{\sigma (r)}))=\mathcal{P}(U_1,\dots,U_r) \, .$$
Equivalently, $i\in \mathcal{P}(U_1,\dots,U_r) \Leftrightarrow \sigma^{-1}(i)\in \mathcal{P}(U_{\sigma (1)},\dots,U_{\sigma (r)})  $.
\end{dfn}

The equivariance condition is quite natural since informally it says that a  thief always chooses (avoids) the same sets, independently of how they are enumerated. An example of a non-equivariant preference arises if a (superstitious) thief always avoids a  safe no.\ 13,
regardless of its  content.

\begin{dfn}
  Given $r$ not necessarily different preferences $\mathcal{P}_1,\dots,\mathcal{P}_r$, an $(m,r)$-admissible family  $\mathcal{U} = \{U_1,\dots, U_r\}$ is envy-free if there exists a permutation $\pi\in S_r$ such that  $\pi(j)\in \mathcal{P}_j(U_1,\dots,U_r)$ for each $j\in [r]$.
\end{dfn}

\noindent (Informally, the thief labeled by $j$ is satisfied if his share is $U_{\pi(j)}$ and he does not envy the other thieves.)

\bigskip
Each partition/preallocation $(x,f)$ determines a unique admissible family $\mathcal{U}$. (This correspondence is clearly not one-to-one since, in the presence of degenerate intervals,  the function $f$ cannot be reconstructed from the family $\mathcal{U}$.) This allows us to express the preferences in the language of the primary  configuration space $\mathcal{C}=\mathcal{C}_{m,r}=(\Delta^{m-1})^{*r}_\Delta$  (Section \ref{sec:new-CS}) and  the preferences of a thief  $j$ can be recorded as a collection of subsets $A_i^j\subseteq \mathcal{C} \,  (i=1,\dots, r)$.

\begin{dfn}\label{def:conf-preferences}
By construction, a point  $(x, f)$ in $\mathcal{C}$ (interpreted as a ``partition/preallocation'') belongs to $A_i^j$ if and only if  $i\in \mathcal{P}_j(U_1,\dots,U_r)$ where   $\mathcal{U} = \{U_1,\dots, U_r\}$
 is the unique $(m,r)$-admissible family  determined by $(x,f)$.
In other words a point $\hat{x} = (x,f)\in \mathcal{C}$ is in $A_i^j$ if and only if  $U_i := \{[x_{k-1}, x_k]  \mid \,  f(k) = i  \mbox{ {\rm and} } x_{k-1}\neq  x_k \}$ is  a set preferred by the thief $j$.
\end{dfn}

\begin{rem}\label{rem:caveat}
  We tacitly assume that the degenerate intervals, which may or may not appear in a partition/preallocation $(x,f)$, are ignored by the thieves. This explains why they are absent from the definition of admissible families (Definition \ref{def:reality}) and the definition of preferences (Definition \ref{def:preference}).
  \end{rem}

Note that Definition \ref{def:conf-preferences} allows us to start with an arbitrary family $\{A_i^j\}_{i,j=1}^r$ of subsets of the configuration space $\mathcal{C}$ (informally called the ``matrix of (topological) preferences''). If this matrix of preferences is {\em proper}, in the sense of the following definition, then it determines a preference function $\mathcal{P}(\mathcal{U}) = \mathcal{P}(U_1,\dots,U_r)$, in the sense of Definition \ref{def:preference}.

\begin{dfn}\label{def:proper}
A collection (matrix) $\{A_i^j\}_{i,j=1}^r$ of  preferences is {\em closed} if all  $A_i^j$ are closed subsets of $\mathcal{C}$. A matrix  $\{A_i^j\}_{i,j=1}^r$  of preferences is {\em proper} if
\[
           \bigcup_{i\in [r]}  A_i^j  = \mathcal{C} \quad \mbox{for each } j\in [r] \, .
  \]
  In agreement with Definition \ref{def:ignores}, a matrix $\{A_i^j\}_{i,j=1}^r$  of preferences is $G$-{\em equivariant} (for a subgroup $G\subseteq S_r$) if
  \[
           \sigma(\hat{x}) \in A_i^j  \quad \Leftrightarrow \quad \hat{x}\in A_{\sigma(i)}^j
  \]
  for each $\hat{x}\in \mathcal{C}, \sigma\in G$ and $i,j\in [r]$. Note that for a fixed $j$ all closed sets $A_i^j$ are homeomorphic.
\end{dfn}

 The following proposition is an immediate consequence of the assumption that the preferences $\mathcal{P}(\mathcal{U}) = \mathcal{P}(U_1,\dots,U_r)$ are  non-empty
sets (Definition \ref{def:preference}).

\begin{prop}\label{prop:cover}
  Suppose that $\{A_i^j\}_{i,j=1}^r$ is a matrix of preferences in the configuration space $\mathcal{C}=\mathcal{C}_{m,r}=(\Delta^{m-1})^{*r}_\Delta$, associated to the preferences $\mathcal{P}(\mathcal{U}) = \mathcal{P}(U_1,\dots,U_r)$. Then $\{A_i^j\}_{i,j=1}^r$ is a covering of the configuration space $\mathcal{C}$, i.e.\ it is a proper family in the sense of Definition \ref{def:proper}.
  \end{prop}

\begin{thm}\label{Thm1}
Let $m=(r-1)(n+1)+1 = N+1$ where  $r\geq 2$ is a prime power and $n\geq 0$. Suppose that $\{\mu_\kappa\}_{\kappa=1}^n$ is a collection of continuous, probability measures on $I = [0,1]$. Let $\{A_i^j\}_{i,j=1}^r$  be a family  of closed subsets (called the matrix of preferences) of the primary configuration space    $\mathcal{C} = (\Delta^N)^{\ast r}_\Delta \cong [r]^{\ast m}$,  which is $S_r$-{\em equivariant} and {\em proper} (Definition \ref{def:proper}).  Then there exist a partition/preallocation  $(x,f)\in \mathcal{C}$ and  a permutation $\pi\in S_r$  such that
\begin{equation}\label{eqn:permutation}
          (x,f)\in \bigcap_{j\in [r]}{A_{\pi(j)}^j}  \, ,  \mbox{ {and}}
\end{equation}
\begin{equation}\label{eqn:fair-fair}
  \mu_\kappa(V_i) = 1/r \quad \mbox{ {for each} } \kappa\in [n]  \mbox{ {and} } i\in [r]
\end{equation}
where
\begin{equation}\label{eqn:V-def}
 V_i = \bigcup_{k\in f^{-1}(i)} [x_{k-1}, x_k]\quad \mbox{ {for all} } i\in [r]   \, .
\end{equation}
 The number  $m=(r-1)(n+1)+1$ is optimal in the sense that such a partition/preallocation  with less than $(m-1)$ cuts in general does not exist.
\end{thm}

\begin{cor}\label{cor:measures}{\rm
  An interesting instance of Theorem \ref{Thm1} arises if, in addition to $n$ measures $\{\mu_\kappa\}_{\kappa=1}^n$ associated to the necklace, each of the thieves $j\in [r]$ has a preference described by a continuous, signed (real-valued) measure $\nu_j$, where he prefers
  $V_i$ if and only if $\nu_j(V_i)\geq \nu_j(V_{i'})$ for each $i'\neq i$. If all  measures $\nu_j$ are the same,  Theorem \ref{Thm1} reduces to Alon's splitting necklace theorem. }
\end{cor}

\proof 
For the proof of Theorem \ref{Thm1} we need to combine the test map used for detecting the fair splittings (Section \ref{subsec:detecting}) with an (equivariant) test map $F : \mathcal{C} \rightarrow \mathbb{R}^r$ which takes into account the preferences of the thieves.

\medskip
The first step is to replace the matrix $\{A_i^j\}_{i,j=1}^r$ of closed sets representing individual preferences by a matrix $\{f_i^j\}_{i,j=1}^r$ of non-negative, real-valued functions such that $\sum_{i=1}^r f_i^j = 1$ for each $j\in [r]$.

\medskip
Recall that the preferences are equivariant in the sense that $\sigma(\hat{x})\in A_i^j   \Leftrightarrow \hat{x} \in A_{\sigma(i)}^j$ for each permutation $\sigma\in G$ and each $\hat{x} = (x,f)\in \mathcal{C}$.  We want this condition to continue to hold for functions in the following form
\begin{equation}\label{eqn:equivariance}
f_{i}^j(\sigma(\hat{x})) = f_{\sigma(i)}^j(\hat{x}) \, .
\end{equation}
This is achieved by replacing $A_i^j $ with slightly larger open preferences $O_i^j$ (in an equivariant way) and by choosing (for each $j\in [r]$) an equivariant partition of unity $\{f_i^j\}_{i=1}^r$ subordinated to the cover $\{O_i^j\}_{i=1}^r$.

 Note that by construction if $f_i^j(\hat{x}) > 0$ then $\hat{x}\in O_i^j$, which means that $\hat{x}$ in this case can be made (by an appropriate choice of $O_i^j$) as close to $A_i^j$ as desired.

\medskip
By averaging $F_i=\frac{1}{r}\sum_{j=1}^{r}f^j_i$ we obtain a vector-valued function  $$F=(F_1,F_2,\dots,F_r):\mathcal{C} \rightarrow \mathbb{R}^r$$
which is also $G$-equivariant.  Let $\widehat{F}  : \mathcal{C} \rightarrow \mathbb{R}^r/D$ be the ($G$-equivariant) map obtained by composing the map $F$ with the projection $\mathbb{R}^r \rightarrow \mathbb{R}^r/D_{\mathbb{R}^r}$, where $D_{V^r}$ stands for the diagonal $D_{V^r} \subset V^r$.

\medskip
By pairing  $\widehat{F}$ with the  map $\widehat{\phi}$ (the ``test map for detecting fair splittings'', eq.\ (\ref{eqn:composite}) in Section  \ref{subsec:detecting})
we finally obtain the $G$-equivariant test map
   $$
 (\widehat{\phi}, \widehat{F}):\mathcal{C} \rightarrow (\mathbb{R}^{n})^r/D_{(\mathbb{R}^{n})^r}\times \mathbb{R}^r/D_{\mathbb{R}^{r}} \cong (\mathbb{R}^{n+1})^r/D_{(\mathbb{R}^{n+1})^r}\, .
 $$
This is very similar to the setting of the  Alon's necklace splitting theorem for $n+1$ measures. Again, by appealing to Volovikov's theorem (Theorem \ref{thm:Vol96}), we conclude that the map  $(\widehat{\phi}, \widehat{F})$ must have a zero.
In other words we have established the existence of a point $\hat{x} = (x,f)\in \mathcal{C}$ such that both $\widehat{\phi}(\hat{x})=0$ and $\widehat{F}(\hat{x})=0$.

\medskip
The point $\hat{x} =(x,f)$ describes a splitting  $I = V_1\cup\dots\cup V_r$ of the necklace, where $V_j = \cup_{f(k)=j} [x_{k-1}, x_k]$,
which is,  as a consequence of $\widehat{\phi}(\hat{x})=0$, a fair partition for each of the measures $\mu_\kappa\, (\kappa=1,\dots, r)$.

\medskip We continue by recalling Gale's original argument \cite{G}. The $(r\times r)$-matrix $(f^j_i(\hat{x}))$ is doubly stochastic, as a consequence  $\widehat{F}(\hat{x})=0$. By the Birkhoff-von Neumann theorem the matrix $(f^j_i(\hat{x}))$ can be expressed as a convex hull of permutation matrices. It follows that there exists a permutation $\sigma\in S_r$ such that $f^j_{\sigma(j)}(\hat{x}) > 0$ for all $j\in [r]$.

\medskip
We conclude that $\hat{x}\in O^j_{\sigma(j)}$ for each $j\in [r]$. In order to find a solution in $A^j_{\pi(j)}$ (for a suitable permutation $\pi $) we take solutions $\hat{x}_k\in O^j_{\sigma_k(j)}$ in a sequence of smaller and smaller neighborhoods of  the original preferences, take a convergent subsequence, and pass to the limit.

\medskip
For the proof of the optimality of the theorem assume that the preferences of the thieves are all the same and that they are dictated by an extra  probability measure $\mu_{n+1}$, meaning that a set $U_i$ is preferred if and only if  $\mu_{n+1}(U_i)\geq 1/r$. In this case Theorem \ref{Thm1} reduces to Alon's original result (Theorem \ref{thm:Alon}) which immediately implies  that in some cases $(r-1)(n+1)$ cuts are necessary.        \qed

\medskip
Theorem \ref{Thm1} is inspired both by Alon's splitting necklace theorem and a theorem about equilibria in mathematical economics, proved independently by Stromquist \cite{Strom} and Woodall \cite{Wood} (see also  of D. Gale \cite{G} and \cite[Theorem 2.2]{AK}). Corollary \ref{cor:measures} is interesting already in the case $n=0$ where it illustrates the envy-free division where some of the thieves may prefer an empty set (see Section \ref{sec:prefer-empty} for more detailed discussion).

\medskip
 The following relative of the ``almost equicardinal fair necklace splitting theorem'' (Theorem \ref{thm:necklace-balanced}) illustrates what we get from Theorem \ref{Thm1} if each of the thieves prefers as small number of pieces as possible. Note that the result obtained by this method is not as strong as Theorem \ref{thm:necklace-balanced}. (This will be rectified by Theorem \ref{Thm2}.)

\begin{cor}\label{cor:relative}
  For given positive integers $r$ and $n$, where $r $ is  a prime power, let  $q$ be an integer such that  $rq > (r-1)n$. Then for any choice of $n$ continuous, probability measures on $[0,1]$ there exists a fair  partition/allocation of the associated necklace with $(r-1)n$ cuts which is almost equicardinal in the sense that
 each thief gets no more than $q$ parts (intervals) of the necklace.
\end{cor}

 \proof Setting $m=(r-1)(n+1)+1$ let us consider, as in Theorem \ref{Thm1}, the corresponding  $(m,r)$-admissible families  $\mathcal{U} = \{U_k\}_{k=1}^r $ and the associated  partitions/preallocations $(x,f)\in \mathcal{C} = \mathcal{C}_{m,r}$.

 \medskip
 Recall that, strictly speaking, the elements of the primary configuration space $\mathcal{C}$ are equivalence classes $[(x,f)]$  (rather than individual  partition/preallo\-cations $(x,f)$) where (Section \ref{sec:primary}) $(x,f_1) \sim (x,f_2)$ if and only if for each $k\in [m]$ if $f_1(k)\neq f_2(k)$ then $I_k = [x_{k-1}, x_k]$ is a degenerate interval.

 \medskip

 Let $\{A_i^j\}_{i,j=1}^r$ be the collection  of preferences (one and the same for all thieves) defined by
  \begin{equation}\label{eqn:one-one}
     [(x,f)] \in A_i^j = A_i  \quad \Leftrightarrow \quad  \vert f_1^{-1}(i)\vert \leq q  \mbox{ {for some}  }  (x,f_1)\sim (x,f)   \, .
 \end{equation}
  In other words a thief $j$ prefers (the content of the safe labeled by) $i$ if and only if the cardinality of the set of all non-degenerate intervals in $f^{-1}(i)$ is at most $q$. Equivalently,  an admissible family  $\mathcal{U} = \{U_k\}_{k=1}^r $ has a representative $(x,f)$ in $A_i^j $ if and only if $\vert U_i\vert \leq q$.

 \medskip
 The matrix  $\{A_i^j\}_{i,j=1}^r$ of topological preferences is clearly closed, equivariant and proper in the sense of Definition \ref{def:proper}. For example it is proper since (by the pigeonhole principle) for each $(x,f)\in \mathcal{C}$ at least one of the sets $f^{-1}(i)$ has cardinality at most $q$.

 \medskip
 By Theorem \ref{Thm1} there exists an element $[(x,f)]\in \cap_{i=1}^r A_i, \, (A_i := A_i^j)$ satisfying the relations (\ref{eqn:permutation}) and (\ref{eqn:fair-fair}).
  It follows from  (\ref{eqn:permutation}) that for each $i=1,\dots, r$ there exists a representative $(x,f_i)$ in the class $[(x,f)]$ such that $\vert f_i^{-1}(i) \vert \leq q$. This, together with (\ref{eqn:fair-fair})  concludes the proof of the corollary.
\qed

\subsection*{Envy-free versions of Theorems \ref{thm:necklace-balanced} and \ref{ThmBinarySpl-cont} }

The method of proof of Theorem \ref{Thm1} is quite general so it is not a surprise that other splitting necklace theorems have extensions with preferences. Here we illustrate the general scheme by ``envy-free'' versions of Theorems \ref{thm:necklace-balanced} and \ref{ThmBinarySpl-cont}.

\medskip
A guiding principle is to start with the corresponding configuration space, describe the preferences $\{A_i^j\}$ as closed subsets of this configuration space and pair the test map with the test map $\widehat{F}$ arising from the preferences.

\medskip
Here is an outline of this procedure for Theorem \ref{thm:necklace-balanced}. Keeping  the same number of cuts as in Theorem \ref{Thm1}, we are now interested in partitions where thieves are allocated almost one and the same number of segments. In particular if $m$ is divisible by $r$, they are all  given  the same number of segments.

\begin{dfn}\label{def:old} Let $m= (r-1)(n+1)+1 = rk + s$ where  $r$ is a prime power, $n\geq 0$  and $k$ and $s$ are the unique integers such that $k\geq 1$ and $0 \leq s < r$.  An  $(m,r)$-admissible family $\mathcal{U} = \{U_i\}_{i=1}^r$  of (collections of segments in) $I=[0,1]$ (see Definition \ref{def:reality}) is {\em  almost equicardinal} if:
\begin{enumerate}
                                   \item Each collection $U_i$ contains at most $k+1$ non-degenerate segments.
                                   \item Not more than $s$ of $U_i$ contain exactly $k+1$ non-degenerate segments.
                                 \end{enumerate}
\end{dfn}
A partition/preallocation $(x,f)\in \mathcal{C} = \mathcal{C}_{m,r}$ is {\em almost equicardinal} if the corresponding $(m,r)$-admissible family $\mathcal{U} = \{U_i\}_{i=1}^r$ is almost equicardinal, where $U_i:= f^{-1}(i)\setminus B_x$ and $B_x  := \{\iota\in [m] \vert \,  x_{\iota-1} =  x_\iota  \}$ is the set (of indices) of degenerate intervals.

\begin{dfn}\label{def:new}
Keeping the initial data $r,n$ and $k,s$ from Definition \ref{def:old}, we set $M=(r-1)(n+2)+1=m+r-1$ and consider  $(M,r)$-admissible families which also satisfy the conditions (1) and (2) (they are called $eq(M,r))$-admissible families for short).

Let $\mathcal{C}_{M,r}$ be the primary configuration space (Section \ref{sec:primary}) with parameters $M$ and $r$ and let $\mathcal{C}' = \mathcal{C}'_{M,r} \subseteq \mathcal{C}_{M,r}$  be the configuration spaces of  classes $[(x,f)]$ such that
$\mathcal{U} = \{U_i\}_{i=1}^r$, where $U_i:= f^{-1}(i)\setminus B'_x$ and $B'_x  := \{\iota\in [M] \vert \,  x_{\iota-1} =  x_\iota  \}$,  is a $eq(M,r))$-admissible family.


\medskip
Alternatively the space $\mathcal{C}'$ can be described as the symmetrized deleted join of $s$ copies of the $k$-skeleton  and $(r-s)$ copies of the $(k-1)$-skeleton  of $\Delta^{M-1}$.

An {\em equicardinal preference }is a collection of subsets  $A_i^j\subset \mathcal{C}' \, (i=1,\dots, r)$. A collection of $r$ preferences is \emph{equicardinally proper} if $\bigcup_{i=1}^r A_i^j=\mathcal{C}'$ for all $j\in [r]$.
\end{dfn}

Note that ``equicardinal properness'' is a less restrictive condition than the ``properness'' in the sense of Definition \ref{def:proper}.

\begin{thm}\label{Thm2}
Let $m=(r-1)(n+1)+1$ where  $r$ is a prime power and  \textcolor{red}{$n\geq 0$}.
For any collection  $\{\mu_\kappa\}_{\kappa=1}^n$ of continuous, probability measures on $[0,1]$ and for any matrix  $(A_i^j)_{i,j=1}^r$ of  equivariant, closed,  equicardinally proper preferences (Definition \ref{def:new}),  there exist a partition/preallocation  $\hat{x} = (x,f)\in \mathcal{C}'$ and  a permutation $\pi\in S_r$  such that
\begin{equation}\label{eqn:permutation-2}
          (x,f)\in \bigcap_{j\in [r]}{A_{\pi(j)}^j}  \, ,  \mbox{ {and}}
\end{equation}
\begin{equation}\label{eqn:fair-fair-2}
  \mu_\kappa(V_i) = 1/r \quad \mbox{ {for each} } \kappa\in [n]  \mbox{ {and} } i\in [r]
\end{equation}
where
\begin{equation}\label{eqn:V-def-2}
 V_i = \bigcup_{\iota\in f^{-1}(i)} [x_{\iota-1}, x_\iota]\quad \mbox{ {for all} } j\in [r]   \, .
\end{equation}
Moreover, the family $\mathcal{U} = \mathcal{U}(\hat{x}) = \{U_i\}_{i=1}^r$ associated to $\hat{x}$ is almost equicardinal in the sense of Definition \ref{def:old}.
\end{thm}

\proof The proof combines the ideas of the proofs of Theorem \ref{Thm1} and Theorem \ref{thm:necklace-balanced}. As in the proof of Theorem \ref{thm:necklace-balanced}, we initially allow   a larger number of cuts $M-1= (r-1)(n+2)$, and then force some of these cuts to be superfluous by an appropriate choice of the configuration space.
The test map (defined on $\mathcal{C}'$) is obtained by pairing  in a single map the test map used in the proof of  Theorem \ref{thm:necklace-balanced} and the map $\widehat{F}$, used in the proof of Theorem \ref{Thm1}.  We use again Volovikov's theorem relying on the fact that the configuration space $\mathcal{C}'$ is $(m-r-1)$-connected,
which is guaranteed by Theorem \ref{thm:jvz-2-symm}.  \qed

\medskip
The method applied in the proofs of Theorems \ref{Thm1} and \ref{Thm2} is quite general and can be used to obtain Stromquist-Woodall-Gale type refinements of other splitting necklace theorems. For example Theorem \ref{thm:general2} admits such an extension. As a variation on a theme here we formulate  an envy-free extension of the binary necklace-splitting theorem (Theorem \ref{ThmBinarySpl-cont}).

\begin{thm}\label{ThmBinarySpl-envy}{\rm (Envy-free binary necklace-splitting theorem)}
If the number of thieves is $r=2^d$, then for each   $n$ continuous probability measures $\mu_1, \dots, \mu_n$ on $[0,1]$ and for  any system of equivariant closed proper preferences, there exists an envy-free binary necklace splitting  with $(r-1)n+r-1$ cuts.
\end{thm}

\section{Envy-free division where players may prefer empty pieces }\label{sec:prefer-empty}

The special cases of Theorems \ref{Thm1} and \ref{Thm2} when there are no measures ($n=0$), are particularly interesting and deserve to be treated separately and compared to earlier envy-free division results.

\medskip

In the ``classical'' setting \cite{G, Strom, Wood}, the players are never satisfied with an empty piece of the ``cake''.  As demonstrated in more recent publications \cite{AK, S-H, MeZe}, under some conditions players may be allowed to choose empty pieces. In the equivariant setting, developed in Section \ref{sec:envy-free}, the choice of empty pieces (empty safes) is allowed by the construction (see Definition \ref{def:preference}), which opens a possibility of applying equivariant methods to problems of this type.
\medskip

\subsection*{Empty pieces may be preferred, no piece may be dropped}

Here we show that currently the most general  result of this type, due to Avvakumov and Karasev \cite[Theorem 4.1]{AK}, can be deduced from Theorem \ref{Thm2}.\footnote{We are grateful to the anonymous referee who kindly pointed to this connection.}

Theorem of Avvakumov and Karasev was established earlier by Meunier and Zerbib \cite[Theorem 1]{MeZe} in the cases when $r$ is a prime or $r=4$, see also Segal-Halevi \cite{S-H} where the result was conjectured and proved in the case $r=3$. It includes the result of Stromquist and Woodall as a special case if the number of players is a prime power.

 \medskip

\medskip
Following (in essence) the setting of \cite[Section 1]{MeZe}, we say that two points $x, x'\in \Delta^{r-1} $, where
\[
   x = (0\leq x_1\leq x_2\leq\dots\leq x_{r-1} \leq 1) \,  \mbox{ {\rm and} }  \, x' = (0\leq x_1'\leq x_2'\leq\dots\leq x_{r-1}' \leq 1)
\]
are {\em partition equivalent}\/ if the corresponding partitions of $[0,1]$ into {non-degene\-ra\-te intervals} are the same. Note that $x$ and $x'$ are partition equivalent if and only if  they are essentially equal as sets in the sense that $\{x_i\}_{i=1}^{r-1}\setminus \{0,1\} = \{x'_i\}_{i=1}^{r-1}\setminus \{0,1\}$.  Less formally, this condition says that $x$ and $x'$ produce the same partition of $[0,1]$ into non-degenerate intervals.

\medskip
A closed subset $X\subseteq \Delta^{r-1}$ is {\em partition balanced} if for each pair $x, x'\in \Delta^{r-1} $ of partition equivalent points
\[
    x \in X \, \Leftrightarrow \, x' \in X \, .
\]
 Note that \emph{partition equivalence} and \emph{partition balanced sets} are closely related to {\em pseudo-equivariance assumptions} from \cite[Section 4.1]{AK}.
Indeed, $x$ and $x'$ are in the same partition equivalence class if and only if one can be obtained from the other by an identification of the form $\sigma_{FGZ}$, described in \cite{AK}.

\medskip
Summarizing, each partition $[0,1] = [0, z_1]\cup [z_1, z_2]\cup\dots\cup [z_s,1]$, where $0<z_1<\dots< z_s<1$ and $0\leq s\leq r-1$, is associated a unique class of partition equivalent points in $\Delta^{r-1}$. Conversely, each point $x\in \Delta^{r-1}$ produces such a partition, which depends only on the partition equivalence class of $x$.

\medskip
A preference function $p^j : \Delta^{r-1} \rightarrow 2^{[r]}\setminus\{\emptyset\}$, in the sense of  \cite[Section 1]{MeZe}, for a given $x\in\Delta^{r-1}$ and the corresponding partition of $[0,1]$, returns a non-empty subset $p^j(x) \subseteq [r]$ where $i\in p^j(x)$ means that the player $j$ is happy to get the interval $[z_{i-1}, z_i]$ if $i\leq s+1$ ($z_0:=0, z_{s+1}:= 1$), or prefers $\emptyset$ if $i>s+1$.

\medskip
We restrict our attention to preference functions $p^j$ such that $X^j_i := \{x\in \Delta^{r-1} \mid i\in p^j(x)\} $ is a partition balanced closed subset of $\Delta^{r-1}$ for each $i\in [r]$.

\medskip
Note that the condition $\cup_{i=1}^r X_i^j = \Delta^{r-1}$, closely related to the {\em full division assumption} from \cite[Section 1]{MeZe} and our {\em properness condition} (Definition \ref{def:proper}), is automatically satisfied since $p^j(x) \neq\emptyset$ for each $x\in \Delta^{r-1}$. Moreover, the converse is also true in the sense that each  matrix $\{X_i^j\}_{i,j=1}^r$ of  partition balanced closed sets, which satisfy the properness condition $\cup_{i=1}^r X_i^j = \Delta^{r-1}$, arises from a collection $\{p^j\}_{j=1}^r$ of preferences in the sense \cite[Section 1]{MeZe}.

\medskip
 The following result addresses the case of the envy-free division problem  \cite{AK, S-H, MeZe} where some players may prefer an ``empty piece'' of the cake.
The condition (1) of the theorem, corresponding to the pseudo-equivariance assumption from \cite{AK}, expresses the idea that if a player prefers  a degenerate interval in a division described by $x\in \Delta^{r-1}$, then she prefers (any) degenerate interval in each partition $x'$, equivalent to $x$.

\begin{thm}{\rm (\cite[Theorem 4.1]{AK})}\label{thm:AK4.1}
Assume that the number $r$ of players in the interval $[0,1]$ partition problem is a prime power. Moreover, assume that the matrix $\{X_i^j\}_{i,j=1}^r$ of closed subsets
of $\Delta^{r-1}$, representing individual preferences of players, satisfies the following conditions:

\begin{enumerate}
  \item The closed set $X_i^j$ is {\em partition balanced} for each $i$ and $j$;
  \item For each $j$ the collection of sets $\{X_i^j\}_{i=1}^r$ is a covering of $\Delta^{r-1}$.
\end{enumerate}
Then there exists a permutation $\pi\in S_r$ such that
\[
    \bigcap_{j\in [r]} X^j_{\pi(j)} \neq \emptyset \, .
\]
In other words there exists a partition of $[0,1]$ into at most $r$ non-degenerate intervals such that  each non-degenerate interval is given to a different player,
the remaining players are not given anything (they are given ``empty pieces''), and this distribution is envy-free from the view point of each of the players.

\end{thm}

\proof
If $n=0$  (Definition \ref{def:old}) then the configuration space $\mathcal{C}'$, described in Definition \ref{def:old}  turns out to be the classical chessboard complex $\Delta_{2r-1,r}$, that is, the complex whose simplices are non-attacking configurations of  rooks in a $[2r-1]\times [r]$ chessboard \cite{Z17, ZV92}.

In order to apply Theorem \ref{Thm2} we use preferences $(X_i^j)_{i,j=1}^r$ defined on $\Delta^{r-1}$, to construct a new set of  preferences $(A_i^j)_{i,j=1}^r$ defined on the configuration space $\mathcal{C}'$.

An element of $\mathcal{C}' \cong \Delta_{2r-1,r}$ is (the equivalence class of) a partition/preallocation $(x,f)$, such that
\[
   0 = x_0 \leq x_1 \leq x_2 \leq \dots \leq x_{2r-2} \leq x_{2r-1} = 1
\]
is a partition of $[0,1]$ into at most $s\leq r$ non-degenerate intervals $\{I^\nu_k\}_{k=1}^s$, where $I^\nu_k := [x_{\nu(k)-1}, x_{\nu(k)}]$ for some strictly monotone function $\nu : [s] \rightarrow [2r-1]$, and $f : [2r-1]\rightarrow [r]$ is a preallocation function.

\medskip
By definition $[(x,f)]\in A_i^j$ if the non-degenerate segment $I_k^\nu$ preferred by $j$ (if such exists) is placed in the safe labelled by $i$ ($f(\nu(k))=i$).
If there are no non-degenerate intervals preferred by $j$ then $s<r$, the function $f\circ\nu : [s] \rightarrow [r]$ is not an epimorphism and as a consequence there exist empty safes. In this case we place $[(x,f)]\in A_i^j$ for each $i\in [r]$ which serves as a label of an empty safe.

\medskip
Less formally,  this construction can be describe in  a form of an algorithm which for an input partition/preallocation $(x,f)$ returns the
labels of the ``safes'' preferred by the player $j$.

\begin{enumerate}
\item Use $x$ to create a partition $y$ of $[0,1]$ with $r-1$ cut points  ($y\in \Delta^{r-1}$), preserving the same collection of non-degenerate intervals $\{I^\nu_k\}_{k=1}^s$. In other words we eliminate $r-1$ superfluous (multiple) cuts. Note that this step can be performed in many different ways.

\item If the preferences $\{X_i^j\}$ dictate the choice of some non-degenerate tiles,  meaning that $y \in X^j_{\nu(k)}$ for some $k$,   find the safes where these intervals belong ($i = f(\nu(k))$) and add them to the preferences of the thief $(x,f)\in A^j_i$.
 Note that the partition-balanced property of $\{X_i^j\}$) implies that no matter how the superfluous cuts are eliminated, the result will be one and the same.
 \item If the preferences $\{X_i^j\}$ dictate to choose a degenerate interval (which occurred  after $r-1$ cuts), observe
 that there necessarily exists an empty safe, since in this case the number of non-degenerate tiles is at most $r-1$. Add this safe to the preferences of the thief.
\end{enumerate}

\medskip
We emphasize again that the fact that $X_i^j$ is a partition balanced subset of $\Delta^{r-1}$ is used in the proof that the set $A_i^j$ is well-defined, in particular that the definition does not depend on the choice of the representative $(x,f)$ in the equivalence class $[x,f]$.

\medskip
It is not difficult to check that the matrix $(A_i^j)_{i,j=1}^r$ of closed subsets of $\mathcal{C}'$ satisfies all conditions of Theorem \ref{thm:necklace-balanced}. So there exist a permutation $\pi\in S_r$ and an element $(x,f)\in \cap_{j\in [r]} A^j_{\pi(j)}$.
It is not difficult to check that the partition of $[0,1]$ into non-degenerate intervals (associated to $x$) and the distribution of pieces, associated to the allocation function $\pi\circ f$, satisfy the conclusion of Theorem \ref{thm:AK4.1}  \qed

\end{document}